\documentclass[11pt]{article} 
\usepackage{amscd,amsmath, amssymb, fancyhdr, epsfig,color}
\usepackage{hyperref}

\numberwithin{equation}{section}

\usepackage{letltxmacro}
\makeatletter
\let\oldr@@t\r@@t
\def\r@@t#1#2{%
\setbox0=\hbox{$\oldr@@t#1{#2\,}$}\dimen0=\ht0
\advance\dimen0-0.2\ht0
\setbox2=\hbox{\vrule height\ht0 depth -\dimen0}%
{\box0\lower0.4pt\box2}}
\LetLtxMacro{\oldsqrt}{\sqrt}
\renewcommand*{\sqrt}[2][\ ]{\oldsqrt[#1]{#2} }
\makeatother

\renewcommand{\leftmark}{{\scriptsize  
Boundness of $b_2$ for hyperk\"ahler manifolds with
  vanishing odd-Betti numbers}}

\newcommand{\tmop}[1]{\ensuremath{\operatorname{#1}}}

\DeclareMathOperator{\Sym}{Sym}

\newcounter{Mycounter}[section]
\newcounter{lemma}[section]
\newcounter{claim}[section]
\newcounter{sublemma}[section]
\newcounter{corollary}[section]
\newcounter{theorem}[section]
\newcounter{conjecture}[section]
\newcounter{proposition}[section]
\newcounter{definition}[section]
\newenvironment{proof}{
{\bf Proof:}}{$\blacksquare$}

\begin{document}
\begin{center}
{\LARGE\bf
Boundness of $b_2$ for hyperk\"ahler manifolds with
  vanishing odd-Betti numbers\\[3mm]
}

Nikon Kurnosov\footnote{The article was prepared within the framework of a subsidy granted to the HSE by the Government of the Russian Federation for the implementation of the Global Competitiveness Program. Nikon Kurnosov is partially supported by grant MK-1297.2014.1., and  RScF grant, project 14-21-00053 dated 11.08.14.}

\end{center}

{\small \hspace{0.05\linewidth}

ABSTRACT. We prove that $b_2$ is bounded for hyperk\"ahler manifolds with vanishing odd-Betti numbers. The explicit upper boundary is conjectured. Following the method described by Sawon we prove that $b_2$ is bounded in dimension eight and ten in the case of vanishing odd-Betti numbers by 24 and 25 respectively.

}




\section{Introduction}

A Riemannian manifold $(M, g)$ is called hyperk\"ahler if it admits a
triple of a complex structures $I, J, K$ satisfying quaternionic relations and
K\"ahler with a respect to $g$. A hyperk\"ahler
manifold is always holomorphically symplectic. By the Yau Theorem \cite{Y}, a
hyperk\"ahler structure exists on a compact complex manifold
if and only if it is K\"ahler and holomorphically symplectic.

\hfill

\begin{definition}
  A compact hyperk\"ahler manifold $M$ is called \textit{simple} if $\pi_1(M)= 0$, $H^{2,0} (M)
  =\mathbb{C}$.
\end{definition}

\hfill

By Bogomolov decomposition theorem \cite{B} any hyperk\"ahler
manifold admits a finite covering which is a product of a torus and several simple hyperk\"ahler manifolds. Huybrechts \cite{H}
has proved finiteness of the number of deformation classes of holomorphic
symplectic structures on each smooth manifold. In most dimensions we only know two examples of simple hyperk\"ahler manifolds due to Beauville \cite{Bea}. This two infinite series are the
Hilbert schemes of $K3$ and the generalized Kummer varieties \cite{Bea}. Except them
two sporadic examples of O'Grady in dimension six and ten are known \cite{O1, O2}. It is known \cite{H} that there are only finitely many families with a given second cohomology lattice, hence the bounds on second Betti number provide restrictions on the number of deformation classes of hyperk\"ahler manifolds. 
There is Beauville Conjecture \cite{Bea1}

\hfill

\begin{conjecture}
The number of
deformation types of compact irreducible hyperk\"ahler manifolds is finite in
any dimension (at least for given $b_2$).
\end{conjecture}

\hfill

In the complex dimension four, Guan \cite{G} proved that the second Betti number of a simple compact hyperk\"ahler manifold of dimension four
is bounded above by 23. Guan also have proved boundary conditions for $b_3$
using Rozansky-Witten invariants \cite{HS}. In dimension six inequality from Rozansky-Witten invariants could also be obtained in dimension six \cite{K}.

In dimension six the boundary conditions for $b_2$ has been obtained by Sawon
\cite{S}. In this paper we generalize Sawon's result on Hodge diamond structure for
dimensions eight and ten.

\hfill

{\bf Acknowledgments.} Author is grateful to his supervisor Misha Verbitsky
for discussions, also author wishes to thank Justin Sawon who kindly describe his method, and Fedor Bogomolov for some ideas.  Author would like
thank Evgeny Shinder and Tom Bridgeland for the possibility to give a talk on
this subject at Sheffield University.

\section{Preliminaries}

In this Section we recall main properties of cohomology groups of hyperk\"ahler manifolds.








There is the following theorem of Verbitsky \cite{V1}

\hfill

\begin{theorem}\label{verb-cohomol}
  Let $M$ be an irreducible hyperk\"ahler manifold of complex dimension a
  $2n$ and let $SH^2$($M$,$\mathbb{C}$)$\subset H^{\ast}(M,\mathbb{C})$ be the
  subalgebra generated by $H^2(M,\mathbb{C})$. Then $SH^2(M,\mathbb{C})$
  = $S^{\ast}H^2(M,\mathbb{C})/\left\langle \alpha^{n + 1} \right| q
  \left( \alpha \right) = 0 \left\rangle \right.$, where $q$ is a Beaville-Bogomolov-Fujiki--form.
\end{theorem}

\hfill

The inclusion $S^n H^2 \left( M \right) \hookrightarrow H^{2 n}$ follows from this Theorem.

\hfill

It is known that in a hyperk\"ahler case there is an action of $\mathfrak{so}(5)$-algebra \cite{V1},
and moreover an action of $\mathfrak{so} ( b_2 + 2,\mathbb{C})$ described by
Looijenga and Lunts \cite{LL}, and Verbitsky \cite{V2}. By this results one can decompose
H$^{even}$(M,$\mathbb{C}$) into irreducible representations for this
$\mathfrak{so}(b_2$ +2,$\mathbb{C}$)-action. The Hodge diamond is the projection onto a
plane of the (higher-dimensional) weight lattice of $\mathfrak{so}(b_2$ +2,$\mathbb{C}$)
so we can figure out the highest weights which lie in octant of Hodge diamond.

Now irreducible representations are isomorphic to $\Lambda^{i}\mathbb{C}^{b_2+2}$. Thus, their dimensions are polynomial in terms of $b_2$. Contribution of $\Lambda^{i}\mathbb{C}^{b_2+2}$ to each even cohomology groups also has dimension polynomial in terms of $b_2$.

\hfill

In early 90-s Salamon \cite{Sa} has used hyperk\"ahler symmetries of Hodge numbers and expression of Euler
characteristic to prove the following equation with Betti
numbers:
\[ n b_{2 n} = 2 \sum_{i = 1}^{2 n} \left( - 1 \right)^i \left( 3 i^2
   - n \right) b_{2 n - i} . \]

\section{Boundary conditions for $b_2$}

In this Section we prove the following

\hfill

\begin{theorem}\label{boundness}
  Let $M$ be a hyperk\"ahler manifold of complex
  dimension $2n$ with vanishing odd-Betti numbers. Then the second Betti
  number is bounded.
\end{theorem}

\hfill

\begin{proof}
  
Consider the $2k$-th Betti number. There is inclusion of $\Sym^k(H^2)$ into $H^{2k}(M)$. Thus, its contribution is polynomial of degree $k$. And also there are contributions of several $\mathfrak{so}(b_2+2,\mathbb{C})$-modules into $H^{2k}(M)$. The dimension of each of them is polynomial in terms of $b_2$, and the contribution of $\Lambda^{i}\mathbb{C}^{b_2+2}$ in $H^{2k}(M)$ is polynomial of degree at most $k-1$. Recall that the highest weights lie in octant of Hodge diamond, and denote them as $c, d$ and etc. For instance, $c$ is dimension of $h^{3,1}_{prim}$.

  Then for all even Betti numbers one could write that
  \[ b_{2 k} = \frac{1}{\left( k - 1 \right) !} \prod^{i = k - 1}_{i = 0}
     \left( b_2 + i \right) + P_k \left( b_2, c, d, e, f, \ldots
     \right), \]
     where first term correspond to the dimension of symmetric power, and the second one -- to all contributions of irreducible $\mathfrak{so}(b_2+2,\mathbb{C})$-modules into $H^{2k}(M)$.
  
  From Salamon's equation
\[ n b_{2 n} = 2 \sum_{i = 1}^{2 n} \left( - 1 \right)^i \left( 3 i^2
   - n \right) b_{2 n - i} . \]
  using vanishing of odd-Betti numbers we obtain the following
  \[ \frac{1}{\left( n - 1 \right) !} \prod^{i = n - 1}_{i = 0} \left( b_2 + i
     \right) + P_n \left( b_2, c, d, e, f, \ldots \right) =\]
     \[= \sum^{j =
     2 n}_{j = 1, j \tmop{even}} \left[ \left( 3 j^2 - n \right)
     \frac{1}{\left( n - j / 2 \right) !} \prod^{i = n - j / 2 - 1}_{i = 0}
     \left( b_2 + i \right) + P_j \left( b_2, c, d, e, f, \ldots
     \right) \right] . \]
  Rearrangering terms in equation above (we put terms corresponding to contributions of symmetric powers of $H^2$ on the left-hand side and all others on the right-hand side):
  \begin{eqnarray*}
    - \frac{1}{\left( n - 1 \right) !} \prod^{i = n - 1}_{i = 0} \left( b_2 +
    i \right) + \sum^{j = 2 n}_{j = 1, j \tmop{even}} \left( 3 j^2 -
    n \right) \frac{1}{\left( n - j / 2 \right) !} \prod^{i = n - j / 2 -
    1}_{i = 0} \left( b_2 + i \right) =\\= P_n \left( b_2, c, d, e, f,
    \ldots \right)
    - \sum^{j = 2 n}_{j = 1, j \tmop{even}} P_j \left( b_2, c, d, e, f, \ldots \right) . &  & 
  \end{eqnarray*}
  Denote by $b_2^l$ the maximal root (in terms of $b_2$) of
  polynomial $P(b_2, n)$ on the left-hand side. Then the polynomial on the left-hand side
  is negative than $b_2 \geqslant b_2^l$. This polynomial
  obviously has some positive roots since it is positive for $b_2 = 0$.

  The numbers $c, d, e, f, \ldots$ are positive since they are numbers of
  generators of irreducible $\mathfrak{so}(b_2 + 2$, $\mathbb{C}$)-modules. The leading
  coefficient of polynomial $Q_n \left( b_2, c, d, e, f, \ldots
  \right)$ on the right-hand side is positive. Thus, this polynomial is
  positive then $b_2 \geqslant b_2^r$, where $b_2^r$ is the maximal root of $Q_n \left( b_2, c, d, e, f, \ldots
  \right)$.\\
  
  Hence we get
  $b_2 \leqslant \max \left[ b_2^l, b_2^r \right]$. In the case if
  right-hand side polynomial $Q_n \left( b_2, c, d, e, f, \ldots
  \right)$ does not have any roots, then $b_2 \leqslant b_2^l$.
\end{proof}

\hfill

\begin{conjecture}\label{poly-conj}
The second Betti number $b_2$ is bounded by the maximal root
of the following polynomial

\[ P \left( b_2, n \right) = - \frac{1}{\left( n - 1 \right) !} \prod^{i = n -
   1}_{i = 0} \left( b_2 + i \right) + \sum^{j = 2 n}_{j = 1, j
   \tmop{even}} \left( 3 j^2 - n \right) \frac{1}{\left( n - j / 2 \right) !}
   \prod^{i = n - j / 2 - 1}_{i = 0} \left( b_2 + i \right), \]
which is denoted as $b_2^l$.
\end{conjecture}

\hfill

To prove this conjecture one need to check that $Q_n \left( b_2, c, d, e, f, \ldots \right)$ is positive than $b_2 \geq b_2^l$. Author does not know the explicit proof of this \ref{poly-conj}.

\hfill

\begin{proposition}
In conditions of \ref{boundness} we have
\[P \left( b_2, n \right) = - \frac{1}{ n !}
\left( \prod^{i = n -
   1}_{i = 3} \left( b_2 + i \right) \right) \cdot (b_2 + 2n) \cdot (b_2^2-21b_2+2-96n),
\]

and 
\[b_2^l = \frac{21+\sqrt{433+96n}}{2}. \]
\end{proposition}

\begin{proof}

We can see that last four terms of $P \left( b_2, n \right)$ are divided by $(b_2+3)$ and factor is 

\[\frac{1}{3}((12n^2-73n+108)b_2^2+3(12n^2-49n+48)b_2+12n(n-1)
\]

Then, we can claim by induction (we omit explicit calculations) that sum of last $k$ terms of $P \left( b_2, n \right)$ is 

\[2 \frac{1}{(k-1)!} \left( \prod^{i = k -
   1}_{i = 3} \left( b_2 + i \right) \right) \cdot(A_k \cdot b_2^2+3B_k \cdot b_2+12n(n-1)),
\]

where $A_k = (12n^2-(73+24(k-4))n+108+60k+24 \frac{(k-4)(k-3)}{2})$, and $B_k = (12n^2-(49+16k)n+48+24k+8\frac{(k-4)(k-3)}{2})$.

Thus, we can write our polynomial $P \left( b_2, n \right)$ in the following form

\[P \left( b_2, n \right) = \frac{1}{\left( n - 1 \right) !} \prod^{i = n - 1}_{i = 0} \left( b_2 + i
     \right) + 2 \frac{1}{(n-1)!} \left( \prod^{i = n -
   1}_{i = 3} \left( b_2 + i \right) \right) \cdot(A_n \cdot b_2^2+3B_n \cdot b_2+12n(n-1)).
\]

After rearranging of common factors we get the statement.
\end{proof}

\section{Dimension eight and ten}

To find explicit boundary condition we shall use the method described by Justin Sawon \cite{S}. In his work he has studied how different $\mathfrak{so}(b_2+2,\mathbb{C})$-modules sit inside Hodge diamond. 

\hfill

\begin{theorem}\label{Dim8}
  Let $M$ be a eight-dimensional hyperk\"ahler manifold with $H^{2
  k + 1} \left( M, \mathbb{C} \right) = 0$. Then $b_2 \leqslant 24$.
\end{theorem}

\begin{proof}

  We will write some explicit splitting of $b_4, b_6,$ and $b_8$ in terms of
  $b_2$ and generators of primitives part of cohomologies.
  
  There is an action of $\mathfrak{so} (b_2 + 2,\mathbb{C})$ on the
  complex cohomology of $M$. Under this action, an element of $H^{3,
  1}_{\tmop{pr}} (M)$ will generate an irreducible $\mathfrak{so} (b_2 + 2,
  \mathbb{C})$-module of dimension $\frac{(b_2 + 2) (b_2 + 1) b_2}{6}$. In
  Hodge diamond this part sits as
  
  {\small
  \begin{center}
    \begin{tabular}{ccccccccccccccccc}
      &  &  &  &  &  &  &  & 0 &  &  &  &  &  &  &  & \\
      &  &  &  &  &  &  & 0 &  & 0 &  &  &  &  &  &  & \\
      &  &  &  &  &  & 0 &  & 0 &  & 0 &  &  &  &  &  & \\
      &  &  &  &  & 0 &  & 0 &  & 0 &  & 0 &  &  &  &  & \\
      &  &  &  & 0 &  & 1 &  & $h$ &  & 1 &  & 0 &  &  &  & \\
      &  &  & 0 &  & 0 &  & 0 &  & 0 &  & 0 &  & 0 &  &  & \\
      &  & 0 &  & 1 &  & $h$ &  & $\frac{h^2 - h + 4}{2}$ &  & $h$ &  & 1 &  & 0
      &  & \\
      & 0 &  & 0 &  & 0 &  & 0 &  & 0 &  & 0 &  & 0 &  & 0 & \\
      0 &  & 0 &  & h &  & $\frac{h^2 - h + 4}{2}$ &  & $\frac{h^3 - 3 h^2 -
      10 h - 60}{6}$ &  & $\frac{h^2 - h + 4}{2}$ &  & $h$ &  & 0 &  & 0\\
      & 0 &  & 0 &  & 0 &  & 0 &  & 0 &  & 0 &  & 0 &  & 0 & \\
      &  & 0 &  & 1 &  & $h$ &  & $\frac{h^2 - h + 4}{2}$ &  & $h$ &  & 1 &  & 0
      &  & \\
      &  &  & 0 &  & 0 &  & 0 &  & 0 &  & 0 &  & 0 &  &  & \\
      &  &  &  & 0 &  & 1 &  & $h$ &  & 1 &  & 0 &  &  &  & \\
      &  &  &  &  & 0 &  & 0 &  & 0 &  & 0 &  &  &  &  & \\
      &  &  &  &  &  & 0 &  & 0 &  & 0 &  &  &  &  &  & \\
      &  &  &  &  &  &  & 0 &  & 0 &  &  &  &  &  &  & \\
      &  &  &  &  &  &  &  & 0 &  &  &  &  &  &  &  & 
    \end{tabular}
  \end{center}
}

  In the same way we have an irreducible $\mathfrak{so} (b_2 + 2,\mathbb{C})$-module
  generated by additional elements of $H^{2, 2}_{\tmop{pr}}$, which sits
  inside the Hodge diamond as
  
  \begin{center}
    \begin{tabular}{ccccccccccccccccc}
      &  &  &  &  &  &  &  & 0 &  &  &  &  &  &  &  & \\
      &  &  &  &  &  &  & 0 &  & 0 &  &  &  &  &  &  & \\
      &  &  &  &  &  & 0 &  & 0 &  & 0 &  &  &  &  &  & \\
      &  &  &  &  & 0 &  & 0 &  & 0 &  & 0 &  &  &  &  & \\
      &  &  &  & 0 &  & 0 &  & 1 &  & 0 &  & 0 &  &  &  & \\
      &  &  & 0 &  & 0 &  & 0 &  & 0 &  & 0 &  & 0 &  &  & \\
      &  & 0 &  & 0 &  & 1 &  & $h$ &  & 1 &  & 0 &  & 0 &  & \\
      & 0 &  & 0 &  & 0 &  & 0 &  & 0 &  & 0 &  & 0 &  & 0 & \\
      0 &  & 0 &  & 1 &  & $h$ &  & $\frac{h^2 - h-4}{2}$ &  & $h$ &  & 1 &  &
      0 &  & 0\\
      & 0 &  & 0 &  & 0 &  & 0 &  & 0 &  & 0 &  & 0 &  & 0 & \\
      &  & 0 &  & 0 &  & 1 &  & $h$ &  & 1 &  & 0 &  & 0 &  & \\
      &  &  & 0 &  & 0 &  & 0 &  & 0 &  & 0 &  & 0 &  &  & \\
      &  &  &  & 0 &  & 0 &  & 1 &  & 0 &  & 0 &  &  &  & \\
      &  &  &  &  & 0 &  & 0 &  & 0 &  & 0 &  &  &  &  & \\
      &  &  &  &  &  & 0 &  & 0 &  & 0 &  &  &  &  &  & \\
      &  &  &  &  &  &  & 0 &  & 0 &  &  &  &  &  &  & \\
      &  &  &  &  &  &  &  & 0 &  &  &  &  &  &  &  & 
    \end{tabular}
  \end{center}

  Recall that $b_4 = \frac{\left( b_2 + 1 \right) b_2}{2} + c b_2 + d,$ where
  $d$ -- part of $H^{2, 2}_{\tmop{pr}}$ that does not come from $H^{3,
  1}_{\tmop{pr}}$
  
  So the first module gives 
 \[ c \left( \frac{h^3 + 3 h^2 - 4 h -36}{6} \right)
  = c \left( \frac{b_2^3 - 3 b_2^2 - 4 b_2 -24}{6} \right),\]
  
  the second one --
  $d \left( \frac{h^2 - h - 4}{2} + 2 h + 2 \right) = d \left( \frac{h^2 + 3 h
  }{2} \right) = d \left( \frac{b_2^2 - b_2 - 2}{2} \right)$.
  
  Definitely, in $H^8$ there are elements which come from those part of $H^6$
  which is not generated by $\tmop{Sym}^3 H^2$ and two $\mathfrak{so} \left( b_2 +
  2, \mathbb{C} \right)$-modules generated by elements of $H^{3,
  1}_{\tmop{pr}}$ and $H^{2, 2}_{\tmop{pr}}$.
  \[ b_6 = \frac{\left( b_2 + 2 \right) \left( b_2 + 1 \right) b_2}{6} + c
     \left( \frac{b_2^2 - b_2 + 2}{2} \right) + d b_2 + e \]
  Then each element of $H^{3, 3}$ part generate $\mathfrak{so} \left( b_2 + 2,
  \mathbb{C} \right)$-module of dimension $\left( b_2 + 2 \right)$. That
  gives in $b_8$ the following term $e \left( b_2 \right)$
  \begin{eqnarray*}
    b_8 \geqslant \frac{\left( b_2 + 3 \right) \left( b_2 + 2 \right) \left(
    b_2 + 1 \right) b_2}{24} + c \left( \frac{b_2^3 - 3 b_2^2 - 4 b_2 -24}{6}
    \right) + d \left( \frac{b_2^2 - b_2 - 2}{2} \right) + e b_2 &  & 
  \end{eqnarray*}

  From Salamon's relation we have
  \begin{eqnarray*}
    8 \cdot \frac{\left( b_2 + 2 \right) \left( b_2 + 1 \right) b_2}{6} + 8 c
    \left( \frac{b_2^2 - b_2 + 2}{2} \right) + 8 d b_2 + 8 e + 44 \left(
    \frac{\left( b_2 + 1 \right) b_2}{2} + c b_2 + d \right)+\\ + 104 b_2  + 188 +
    b_7 - 71 b_3 - 23 b_5 \geqslant \\ \geqslant 2\frac{\left( b_2 + 3 \right) \left( b_2 +
    2 \right) \left( b_2 + 1 \right) b_2}{24} + 2c \left( \frac{b_2^3 - 3 b_2^2 - 4 b_2 -24}{6} \right) + 2d \left( \frac{b_2^2 - b_2 - 2}{2} \right) + 2e b_2
    &  & 
  \end{eqnarray*}

  Then after rearranging we get
  
  \begin{eqnarray*}
- 2 \frac{\left( b_2 + 3 \right) \left( b_2 + 2 \right) \left( b_2
    + 1 \right) b_2}{24} + 8 \cdot \frac{\left( b_2 + 2 \right) \left( b_2 + 1 \right) b_2}{6} + 44
    \left( \frac{\left( b_2 + 1 \right) b_2}{2} \right) + 104 b_2 + 188 + b_7 \geqslant \\
    \geqslant c \left( \frac{b^3_2 - 15 b_2^2 - 124 b_2 -
    48}{3} \right) + d \left( b_2^2 -  9b_2 - 46 \right) + 2 e
    \left( b_2 - 4 \right) + 71 b_3 + 23 b_5               
  \end{eqnarray*}
  
  The left-part is $\frac{- b_2^4 + 10 b_2^3 + 301 b_2^2 + 530 b_2 + 2256}{12}
  + b_7$.\\
  
  Now recall that all odd Betti numbers are zero. Then the left-hand side is
  negative for $b_2 \geqslant 24$, although the right-hand side is positive.  
\end{proof}

\hfill

{\bf{Remark:}} The second Betti number $b_2$ is at most 24 for a sufficiently small $b_7$.
Indeed, for the proof of \ref{Dim8} we use the fact that
\[ F \left( b_2 \right) + b_7 := \frac{- 2 b_2^4 + 20 b_2^3 + 602 b_2^2
   + 1060 b_2 + 4512}{24} + b_7 \]
is negative than $b_2 \geqslant 25$. This is also true for $b_7 \leqslant
\left| F \left( b_2 \right)\vert_{b_2 = 25} \right| = 1281$.

\hfill

\begin{theorem}\label{Dim10}
  Let $M$ be a ten-dimensional hyperk\"ahler manifold with $H^{2 k
  + 1} \left( M, \mathbb{C} \right) = 0$. Then $b_2 \leqslant 25$.
\end{theorem}

\begin{proof}
  The proof is very similar to the previous one. We could see that they are
  contributions from $H^{3, 1}_{\tmop{pr}}, H^{2, 2}_{\tmop{pr}},
  H^{3, 3}_{\tmop{pr}}, H^{4, 4}_{\tmop{pr}} .$ They are isomorphic to
  $\bigwedge^4 \mathbb{C}^{b_2 + 2}, \bigwedge^3 \mathbb{C}^{b_2 +
  2}, \bigwedge^2 \mathbb{C}^{b_2 + 2}$, and $\mathbb{C}^{b_2 + 2}$ as
  irreducible $\mathfrak{so} \left( b_2 + 2,\mathbb{C} \right) - \tmop{modules}$.
  
  The explicit calculations like in the case of dimension eight give us the
  following  
  \[ \frac{1}{60}(b_2+3)(b_2+4)(b_2+10) ( - b_2^2- 21b_2 + 118) = Q \left( b_2, c, d, e, f, g \right),\]
  
  Polynomial Q($b_2$) is positive for $b_2 \geqslant 26$ since $c, d, e, f, g$ are positive constants defined above ($f$ sits in
  $H^{4, 4}$-part, $g$ generates $H^{5, 5}$). The left-hand side is negative
  for $b_2 \geqslant 26$. Hence, $b_2 \leqslant 25$.
\end{proof}

\hfill

{\small
\noindent {\sc Nikon Kurnosov\\
Laboratory of Algebraic Geometry and its applications,\\
National Research University Higher School of Economics\\
7 Vavilova Str., Moscow, Russia, 117312;\\
Independent University of Moscow,\\
11 Bol.Vlas'evskiy per., Moscow, Russia, 119002}\\
\it  nikon.kurnosov@gmail.com 
}

\end{document}